\documentclass[10pt]{amsart}

\usepackage{amsmath,amsfonts, latexsym,graphicx, amssymb, mathtools, enumerate, mdwlist, amscd}

\usepackage{hyperref}
\hypersetup{colorlinks=true, urlcolor=blue, citecolor=blue, linkcolor=blue}

\newcommand{\C}{{\mathbb C}}
\newcommand{\Z}{{\mathbb Z}}

\newcommand{\Q}{{\mathbb Q}}

\newcommand{\arrow}[1]{\stackrel{#1}{\longrightarrow}}
\newcommand{\Adot}{\mathbf A^\bullet}
\newcommand{\Bdot}{\mathbf B^\bullet}

\newcommand{\vdual}{{\mathcal D}}

\newcommand{\var}{\mathbf{var}}
\newcommand{\can}{\mathbf{can}}

\newtheorem{defn0}{Definition}[section]
\newtheorem{prop0}[defn0]{Proposition}
\newtheorem{conj0}[defn0]{Conjecture}
\newtheorem{thm0}[defn0]{Theorem}
\newtheorem{lem0}[defn0]{Lemma}
\newtheorem{corollary0}[defn0]{Corollary}
\newtheorem{example0}[defn0]{Example}
\newtheorem{remark0}[defn0]{Remark}
\newtheorem{question0}[defn0]{Question}
\newtheorem{exercise0}[defn0]{Exercise}

\newenvironment{thm}{\begin{thm0}}{\end{thm0}}
\newenvironment{lem}{\begin{lem0}}{\end{lem0}}

\newenvironment{rem}{\begin{remark0}\rm}{\end{remark0}}

\newcommand{\thmref}[1]{Theorem~\ref{#1}}
\newcommand{\lemref}[1]{Lemma~\ref{#1}}

\newcommand{\remref}[1]{Remark~\ref{#1}}

\title{Verdier dualizing and the variation map}

\subjclass[2010]{32S25, 32S15, 32S55}

\author{David B. Massey}

\date{}

\begin{document}

\begin{abstract} We describe how the canonical and variation maps between the shifted nearby and vanishing cycles interact with Verdier dualizing.
\end{abstract}

\maketitle




\section{Introduction} 

Let $X$ be a reduced complex analytic space and let $f:X\rightarrow \C$ be a nowhere locally constant analytic function. Let $j: V(f)\hookrightarrow X$ and $i: X-V(f)\hookrightarrow X$ denote the inclusions. (We initially learned about the derived category from Goresky and MacPherson's ``Intersection Homology II'', \cite{inthom2}. The convention of $j$ being a closed inclusion, while $i$ is an open inclusion -- which is reversed from what many people use -- is theirs.)

\medskip

We fix a base ring, $R$, which is a commutative, regular, Noetherian ring, with finite Krull dimension (e.g., $\Z$, $\Q$, or $\C$). When we write that $\Adot$ is complex of sheaves on a topological space, $X$, we mean that $\Adot$ is an object in $D^b(X)$, the derived category of bounded complexes of sheaves of $R$-modules on $X$. When $X$ is complex analytic, we may also require that $\Adot$ is (complex) constructible, and write $\Adot\in D^b_c(X)$. We remind the reader that constructibility includes the assumption that the stalks of all cohomology sheaves are finitely-generated $R$-modules (so that, by our assumption on $R$, each stalk complex $\Adot_x$, for $x\in X$, is perfect, i.e., is quasi-isomorphic to a bounded complex of finitely-generated projective $R$-modules).

\smallskip

Let $\vdual=\vdual_Z$ denote the Verdier dualizing operator on $D^b_c(Z)$. We will always write simply $\vdual$, since the relevant space will always be clear. 

\smallskip

As we will always work in the derived category, we omit the $R$ at the beginning of right derived functors; all of our functors are derived. We also write $f^*$ where some authors would write $f^{-1}$.

\smallskip

We define the vanishing cycles as in Lemma 1.3.2 of \cite{schurbook}, or following Exercise VIII.13 of \cite{kashsch} (but reversing the inequality, and using a different shift); this, of course, means that by definition, the stalk cohomology of vanishing cycles is the hypercohomology of a small ball modulo the Milnor fiber.

\smallskip

We use $T_f$ and $\widetilde{T}_f$ for the monodromy natural automorphisms on the shifted nearby and vanishing cycles $\psi_f[-1]$ and $\phi_f[-1]$, respectively. We also have the standard {\it canonical} and {\it variation} natural transformations, $\can$ and $\var$:
$$\psi_f[-1]\xrightarrow{\can} \phi_f[-1]\hskip 0.2in\textnormal{ and }\hskip 0.2in \phi_f[-1]\xrightarrow{\var}\psi_f[-1],$$
where 
$$
\big({\operatorname{id}}-T_f\big)_{\Adot} = \var_{\Adot}\circ\can_{\Adot}\hskip 0.2in \textnormal{ and }\hskip 0.2in\big({\operatorname{id}}-\widetilde T_f\big)_{\Adot} = \can_{\Adot}\circ\var_{\Adot}.
$$

\vskip 0.2in

In \cite{natcommute}, we produced a natural isomorphism
\begin{equation}\label{eq:vandual}
\phi_f[-1]\circ\vdual \xrightarrow[\cong]{\epsilon}\vdual\circ\phi_f[-1].
\end{equation} 
We then used the natural isomorphism $\epsilon$ to produce a natrual isomorphism
\begin{equation}\label{eq:neardual}
\psi_f[-1]\circ\vdual \xrightarrow[\cong]{\delta}\vdual\circ\psi_f[-1].\end{equation}

However, something which we did {\bf not} do in \cite{natcommute} was discuss/prove in what sense $\can$ and $\var$ are dual to each other. That is the point of this paper.

\vskip 0.2in 

We prove:

\begin{thm}\label{thm:mainintro} The following diagram commutes:
$$
\begin{CD}
\psi_f[-1]\vdual\Adot@>\can_{\vdual\Adot}>>\phi_f[-1]\vdual\Adot\\
@VV\delta_{\Adot}V  @VV\epsilon_{\Adot}V\\
\vdual\psi_f[-1]\Adot @>\vdual(\var_{\Adot})>>\vdual\phi_f[-1]\Adot.
\end{CD}
$$
\end{thm}

\vskip 0.2in 

However, one must be careful now. It is tempting to assume that one can also show, in an analogous fashion, that
\begin{equation}\label{eq:noncomm}
\begin{CD}
\phi_f[-1]\vdual\Adot @>\var_{\vdual\Adot}>>\psi_f[-1]\vdual\Adot\\
 @VV\epsilon_{\Adot}V  @VV\delta_{\Adot}V\\
\vdual\phi_f[-1]\Adot @>\vdual(\can_{\Adot})>>\vdual\psi_f[-1]\Adot 
\end{CD}
\end{equation}

\medskip

\noindent commutes; however, {\bf this is false}. To make the diagram above commute, one must replace the natural isomorphism $\delta$ with a different natural isomorphism $\hat\delta$.

\vskip 0.3in

This paper is organized as follows:

\medskip

\noindent

In Section 2, we prove a short, technical, lemma that we shall need; we suspect that this lemma is known, but we can find no reference. In Section 3, we prove the main theorem,  \thmref{thm:main}, (as stated above in \thmref{thm:mainintro}). In Section 4, we discuss the failure of the commutation of \eqref{eq:noncomm} and show how to define $\hat\delta$ to ``fix'' this issue. 

\medskip

\section{A Basic Lemma}

In this section,  we will prove a general lemma, and then apply it in later sections to the situation of the introduction. 

\bigskip

Let $k: Y\rightarrow Z$ be a continuous map of locally compact topological spaces such that $k_!$ has finite cohomological dimension (e.g., $k$ is the inclusion of a locally closed subset).  

\bigskip

For all $\Bdot\in D^b(Y)$ and $\Adot\in D^b(Z)$ (note that we are not assuming constructibility here), there are canonical isomorphisms of morphism groups
$$
\operatorname{Hom}(k^*\Adot, \Bdot)\xrightarrow[\cong]{M_{\Adot, \Bdot}}\operatorname{Hom}(\Adot, k_*\Bdot)
$$
and
$$
\operatorname{Hom}(\Bdot, k^!\Adot)\xrightarrow[\cong]{N_{\Adot, \Bdot}}\operatorname{Hom}(k_!\Bdot, \Adot);
$$
see \cite{kashsch}, Proposition II.2.6.4 and Theorem III.3.1.5. Replacing $\Bdot$ with $k^*\Adot$ in the first isomorphism and $\Bdot$ with $k^!\Adot$ (which, of course, equals $k^*\Adot$), we obtain isomorphisms
$$
\operatorname{Hom}(k^*\Adot, k^*\Adot)\xrightarrow[\cong]{M_{\Adot}}\operatorname{Hom}(\Adot, k_*k^*\Adot)
$$
and
$$
\operatorname{Hom}(k^!\Adot, k^!\Adot)\xrightarrow[\cong]{N_{\Adot}}\operatorname{Hom}(k_!k^!\Adot, \Adot).
$$

\smallskip

The images of the identity morphisms yield the standard natural transformations
$$
\operatorname{id}\arrow{\alpha} k_*k^*\hskip 0.2in \textnormal{ and }\hskip 0.2in   k_!k^!\arrow{\beta} \operatorname{id},
$$
that is, $\alpha_{\Adot}=M_{\Adot}(\operatorname{id}_{k^*\Adot})$ and $\beta_{\Adot}=N_{\Adot}(\operatorname{id}_{k^!\Adot})$.

\bigskip

We return to the situation and notation of the introduction. In the lemma below, we need constructibility of the complexes. Note that $i_*$ and $i_!$ are applied only to $i^*$ and $i^!$ of constructible complexes on $X$ and so yield constructible complexes.

\smallskip

\begin{lem}\label{lem:abd} There is a natural isomorphism $$\vdual i_*i^*\xrightarrow[\cong]{\gamma} i_!i^!\vdual$$ such that, for all $\Adot\in D^b_c(X)$,
$$
\beta_{\vdual\Adot}=\vdual(\alpha_{\Adot})\circ \gamma_{\Adot}^{-1}.
$$

\medskip
\end{lem}
\begin{proof} Suppose that we have  $\Adot, \Bdot\in D_c^b(X)$.
Then, we have natural isomorphisms

$$
\operatorname{Hom}(\vdual i_*i^*\Adot, \vdual \Bdot)\cong \operatorname{Hom}(\Bdot, i_*i^*\Adot, )\cong  \operatorname{Hom}(i^*\Bdot, i^*\Adot)\cong \operatorname{Hom}(\vdual i^*\Adot, \vdual i^*\Bdot)\cong
$$
$$
 \operatorname{Hom}(i^!\vdual \Adot, i^!\vdual \Bdot)\cong \operatorname{Hom}(i_!i^!\vdual\Adot, \vdual \Bdot),
$$

\medskip

\noindent where the isomorphisms, in order, follow from  Proposition III.3.4.6 of \cite{kashsch}, Proposition II.2.6.4 of  \cite{kashsch}, Proposition III.3.4.6 of \cite{kashsch}, Theorem V.10.11 of \cite{boreletal}, and Proposition III.3.1.10  of \cite{kashsch}. As this is true for all $\Bdot$, it follows that there is a natural isomorphism  $$\vdual i_*i^*\xrightarrow[\cong]{\gamma} i_!i^!\vdual$$ inducing the isomorphism from $\operatorname{Hom}(\vdual i_*i^*\Adot, \vdual \Bdot)$ to $\operatorname{Hom}(i_!i^!\vdual\Adot, \vdual \Bdot)$.

Replacing $\Bdot$ with $\Adot$, tracing through the $\operatorname{Hom}$ isomorphisms, and using the definitions of $\alpha$ and $\beta$, one obtains immediately that 
$$
\beta_{\vdual\Adot}=\vdual(\alpha_{\Adot})\circ \gamma_{\Adot}^{-1}.
$$
\end{proof}

\medskip

\begin{rem} It is easy to find the existence of a natural isomorphism $\vdual i_*i^*\xrightarrow[\cong]{\gamma} i_!i^!\vdual$ in the literature. The point of the lemma above is how $\gamma$ interacts with Verdier dualizing.
\end{rem}

\medskip

\section{The Main Theorem}

In this section, we will prove our main theorem as stated in \thmref{thm:mainintro}. It may seem as though we define an excessive number of maps, but we have found no simpler presentation. We continue with our notation from the introduction.

\bigskip

Consider the two standard natural distinguished triangles
$$
\rightarrow j_*j^*[-1]\rightarrow\psi_f[-1]\xrightarrow{\can}\phi_f[-1]\xrightarrow{[1]}
$$
and
$$
\rightarrow\phi_f[-1]\xrightarrow{\var}\psi_f[-1]\rightarrow j_!j^![1]\xrightarrow{[1]}.
$$

\bigskip

Recall the natural transformations $\alpha$ and $\beta$ from the previous section. By naturality, we immediately conclude:

\bigskip

\noindent Let $\Adot\in D^b_c(X)$. As $\psi_f[-1]\Adot$ depends only on $i^*\Adot$ (and $f$), there are trivial natural isomorphisms
$$
\psi_f[-1]\xrightarrow[\cong]{\pi}\psi_f[-1]i_!i^!\hskip 0.2in\textnormal{ and }\hskip 0.2in \psi_f[-1]i_*i^*\xrightarrow[\cong]{\rho}\psi_f[-1]
$$
given by $\pi_{\Adot}=(\psi_f[-1](\beta_{\Adot}))^{-1}$ and $\rho_{\Adot}=(\psi_f[-1](\alpha_{\Adot}))^{-1}$.

\medskip

 Furthermore, as $j_*j^*[-1]i_!i^!=0$ and $j_!j^![1]i_*i^*=0$, we have natural isomorphisms
$$
\psi_f[-1]i_!i^!\xrightarrow[\cong]{\can_{i_!i^!}}\phi_f[-1]i_!i^!\hskip 0.2in\textnormal{ and }\hskip 0.2in \phi_f[-1]i_*i^*\xrightarrow[\cong]{\var_{i_*i^*}}\psi_f[-1]i_*i^*.
$$

\bigskip

\noindent We let 
$$\sigma_{\Adot}:=\can_{i_!i^!\Adot}\circ\pi_{\Adot}\hskip 0.2in\textnormal{ and }\hskip 0.2in \tau_{\Adot}:=\rho_{\Adot}\circ\var_{i_*i^*\Adot},$$ so that $\sigma$ is a natural isomorphism from $\psi_f[-1]$ to $\phi_f[-1]i_!i^!$ and $\tau$ is a natural isomorphism from $\phi_f[-1]i_*i$ to $\psi_f[-1]$.

\bigskip

\begin{lem}\label{lem:2squares} The following diagrams commute:

$$
\begin{CD}
\phi_f[-1]\Adot @>\phi_f[-1](\alpha_{\Adot})>>\phi_f[-1]i_*i^*\Adot \\
@VV\var_{\Adot}V    @VV\var_{i_*i^*\Adot}V\\
\psi_f[-1]\Adot @>\rho_{\Adot}^{-1}=\psi_f[-1](\alpha_{\Adot})>>\psi_f[-1]i_*i^*\Adot
\end{CD}
$$
and

$$
\begin{CD}
\psi_f[-1]i_!i^!\Adot @>\pi_{\Adot}^{-1}=\psi_f[-1](\beta_{\Adot})>>\psi_f[-1]\Adot \\
@VV\can_{i_!i^!\Adot}V    @VV\can_{\Adot}V\\
\phi_f[-1]i_!i^!\Adot @>\phi_f[-1](\beta_{\Adot})>>\phi_f[-1]\Adot,
\end{CD}
$$

\bigskip

\noindent where $\can_{i_!i^!\Adot}$, $\var_{i_*i^*\Adot}$, $\psi_f[-1](\alpha_{\Adot})$, and $\psi_f[-1](\beta_{\Adot})$ are isomorphisms.

\medskip

Therefore, 
$$\var_{\Adot}=\rho_{\Adot}\circ \var_{i_*i^*\Adot} \circ\phi_f[-1](\alpha_{\Adot})=\tau_{\Adot}\circ \phi_f[-1](\alpha_{\Adot})$$ 
and 
$$
\can_{\Adot} = \phi_f[-1](\beta_{\Adot})\circ \can_{i_!i^!\Adot}\circ \pi_{\Adot} =\phi_f[-1](\beta_{\Adot})\circ\sigma_{\Adot} .
$$
\end{lem}
\begin{proof} These are immediate from the naturality of $\can$ and $\var$.
\end{proof}

\vskip 0.2in

Recall the definition of $\gamma$ from \lemref{lem:abd}. In \cite{natcommute}, the natural isomorphism $\delta$ is defined in terms of $\epsilon$ by, for all $\Adot\in D^b_c(X)$,
$$
\delta_{\Adot} =\vdual(\tau^{-1}_{\Adot})\circ\epsilon_{i_*i^*\Adot}\circ\phi_f[-1](\gamma^{-1}_{\Adot})\circ\sigma_{\vdual\Adot}.
$$

\medskip

With this definition, it is easy to show:

\medskip

\begin{lem}\label{lem:2moresquares} The following diagram commutes:
$$
\begin{CD}
\psi_f[-1]\vdual\Adot@>\phi_f[-1](\gamma_{\Adot}^{-1})\circ\sigma_{\vdual\Adot}>>\phi_f[-1]\vdual i_*i^*\Adot @>\phi_f[-1](\vdual(\alpha_{\Adot}))>>\phi_f[-1]\vdual\Adot\\
@VV\delta_{\Adot}V @VV\epsilon_{i_*i^*\Adot}V  @VV\epsilon_{\Adot}V\\
\vdual\psi_f[-1]\Adot@>\vdual(\tau_{\Adot})>>\vdual\phi_f[-1]i_*i^*\Adot @>\vdual(\phi_f[-1]\alpha_{\Adot})>>\vdual\phi_f[-1]\Adot .
\end{CD}
$$
\end{lem}
\begin{proof} The square on the left commutes by the definition of $\delta$ and the square on the right commutes by the naturality of $\epsilon$.
\end{proof}

\bigskip

Finally, with all of our lemmas, it is simple to prove:
\begin{thm}\label{thm:main}
The following diagram commutes:
$$
\begin{CD}
\psi_f[-1]\vdual\Adot@>\can_{\vdual\Adot}>>\phi_f[-1]\vdual\Adot\\
@VV\delta_{\Adot}V  @VV\epsilon_{\Adot}V\\
\vdual\psi_f[-1]\Adot @>\vdual(\var_{\Adot})>>\vdual\phi_f[-1]\Adot,
\end{CD}
$$

\bigskip

\noindent i.e., $\can_{\vdual\Adot}=\epsilon_{\Adot}^{-1}\circ\vdual(\var_{\Adot})\circ\delta_{\Adot}$.
\end{thm}
\begin{proof} This follows from the outside edges of the diagram in \lemref{lem:2moresquares}. The composition along the bottom row, combined with \lemref{lem:2squares}, yields
$$
\vdual(\var_{\Adot})=\vdual\big(\tau_{\Adot}\circ \phi_f[-1](\alpha_{\Adot})\big)=\vdual(\phi_f[-1](\alpha_{\Adot}))\circ\vdual(\tau_{\Adot}).
$$

Composition along the top row yields
$$
\phi_f[-1](\vdual(\alpha_{\Adot}))\circ \phi_f[-1](\gamma_{\Adot}^{-1})\circ\sigma_{\vdual\Adot}=\phi_f[-1]\big(\vdual(\alpha_{\Adot})\circ\gamma^{-1}_{\Adot}\big)\circ\sigma_{\vdual\Adot},
$$
which, by \lemref{lem:abd} and \lemref{lem:2squares}, is equal to
$$
\phi_f[-1](\beta_{\vdual\Adot})\circ\sigma_{\vdual\Adot}=\can_{\vdual\Adot}.
$$
\end{proof}

\bigskip

\begin{rem}\label{rem:formula} Note that the equality in \thmref{thm:main} immediately yields a formula for $\var_{\vdual\Adot}$ (using twice that $\vdual\vdual$ is naturally isomorphic to the identity). One applies the equality in \thmref{thm:main} to $\vdual\Adot$ in place of $\Adot$ and then dualizes both sides to obtain:

$$
\can_{\vdual\Adot}=\epsilon_{\Adot}^{-1}\circ\vdual(\var_{\Adot})\circ\delta_{\Adot}.
$$

$$
\can_{\Adot}=\epsilon_{\vdual\Adot}^{-1}\circ\vdual(\var_{\vdual\Adot})\circ\delta_{\vdual\Adot}.
$$

$$
\vdual(\can_{\Adot})=\vdual(\delta_{\vdual\Adot})\circ\var_{\vdual\Adot}\circ \vdual(\epsilon_{\vdual\Adot}^{-1}).
$$

$$
\var_{\vdual\Adot}= \vdual(\delta^{-1}_{\vdual\Adot})\circ\vdual(\can_{\Adot})\circ \vdual(\epsilon_{\vdual\Adot}).
$$
\end{rem}

\medskip

\section{A non-commuting diagram}

\medskip

In \thmref{thm:main}, we proved that the right-hand square of the following diagram commutes:

\smallskip

\begin{equation}\label{eq:notcomm}
\begin{CD}
\phi_f[-1]\vdual\Adot @>\var_{\vdual\Adot}>>\psi_f[-1]\vdual\Adot@>\can_{\vdual\Adot}>>\phi_f[-1]\vdual\Adot\\
@VV\epsilon_{\Adot}V  @VV\delta_{\Adot}V  @VV\epsilon_{\Adot}V\\
\vdual\phi_f[-1]\Adot @>\vdual(\can_{\Adot})>>\vdual\psi_f[-1]\Adot @>\vdual(\var_{\Adot})>>\vdual\phi_f[-1]
\end{CD}
\end{equation}

\bigskip

\noindent One might hope that, in an analogous way, we could prove that the left-hand side also commutes. This is not the case; {\bf the left-hand side does not commute}. The fact that not both of the squares above can commute was pointed out to us by J\"org Sch\"urmann. We explain this now.

\bigskip

From the construction of $\epsilon$, it is relatively simple to show that there is a commutative diagram

$$
\begin{CD}
\phi_f[-1]\vdual\Adot @>(\widetilde T_f^{-1})_{\vdual\Adot}>>\phi_f[-1]\vdual\Adot \\
@VV\epsilon_{\Adot}V    @VV\epsilon_{\Adot}V\\
\vdual\phi_f[-1]\Adot @>\vdual((\widetilde T_f)_{\Adot})>>\vdual\phi_f[-1]\Adot .
\end{CD}
$$

\bigskip

From this, it follows that, {\bf if \eqref{eq:notcomm} were commutative}, then the outside edges would tell us that
$$
\big({\operatorname{id}}-\widetilde T_f\big)_{\vdual\Adot} = \can_{\vdual\Adot}\circ\var_{\vdual\Adot}=\epsilon^{-1}_{\Adot}\circ\vdual(\can_{\Adot}\circ\var_{\Adot})\circ\epsilon_{\Adot} =
$$
$$
\epsilon^{-1}_{\Adot}\circ\vdual\big(\big({\operatorname{id}}-\widetilde T_f\big)_{\Adot}\big)\circ\epsilon_{\Adot}={\operatorname{id}}_{\vdual\Adot}-\epsilon^{-1}_{\Adot}\circ\vdual\big((\widetilde T_f)_{\Adot}\big)\circ\epsilon_{\Adot}=
$$
$$
{\operatorname{id}}_{\vdual\Adot}-(\widetilde T_f^{-1})_{\vdual\Adot}
$$

\medskip

\noindent As $\widetilde T_f^{-1}$ is, in general, not equal to $\widetilde T_f$, we conclude that \eqref{eq:notcomm} does not commute.

\bigskip

In fact, it is easy to see that it was our choice of the isomorphism $\delta$  which made the right side of \eqref{eq:notcomm} commute and not the left. If, however, we define the natural isomorphism
$$
\hat\delta_{\Adot}:= \vdual(\sigma_{\Adot})\circ\epsilon_{i_!i^!\Adot}\circ\phi_f[-1](\vdual(\gamma^{-1}_{\vdual\Adot}))\circ\tau^{-1}_{\vdual\Adot},
$$
then the analogous argument to what we used to prove \thmref{thm:main} shows that we have a commutative diagram
$$
\begin{CD}
\phi_f[-1]\vdual\Adot @>\var_{\vdual\Adot}>>\psi_f[-1]\vdual\Adot\\
 @VV\epsilon_{\Adot}V  @VV\hat\delta_{\Adot}V\\
\vdual\phi_f[-1]\Adot @>\vdual(\can_{\Adot})>>\vdual\psi_f[-1]\Adot .
\end{CD}
$$

\medskip

Note that this gives a different formula for $\var_{\vdual\Adot}$ than what we found at the end of \remref{rem:formula}. We have not found a direct way of passing from one formula to the other; we just know that they both equal $\var_{\vdual\Adot}$.

\medskip

\bibliographystyle{plain}

\bibliography{Masseybib}

\end{document}